\documentclass[11pt]{article}

\usepackage[margin=1in]{geometry}
\usepackage[margin=0.05\textwidth]{caption}
\usepackage{subcaption}
\usepackage{mathtools,amsthm,amssymb, graphicx, multicol, array, fancyhdr, enumerate, titlesec, xcolor, mathrsfs, tikz, hyperref}

\title{Bernoulli Randomness and Biased Normality}
\date{July 2020\footnote{Last updated on \today}}
\author{Andrew DeLapo\thanks{This work was the author's senior honors thesis which was completed in the Department of Mathematics at the University of California, Berkeley, supervised by Professor Theodore Slaman.}\\\texttt{adelapo@berkeley.edu}}
	
\theoremstyle{definition}
\newtheorem{theorem}{Theorem}[section]

\newtheorem{lemma}[theorem]{Lemma}
\newtheorem{corollary}[theorem]{Corollary}
\newtheorem{definition}[theorem]{Definition}
\newtheorem{example}[theorem]{Example}
\newtheorem{construction}[theorem]{Construction}
\newtheorem*{conjecture}{Conjecture}

\numberwithin{equation}{section}

\definecolor{purpl}{RGB}{131,0,175}
\newcommand{\N}{\mathbb{N}}
\newcommand{\Z}{\mathbb{Z}}
\newcommand{\R}{\mathbb{R}}

\renewcommand{\bar}{\overline}

\newcommand{\abs}[1]{\left| #1 \right|}
\newcommand{\card}[1]{\left| #1 \right|}
\newcommand{\dbracket}[1]{\left[\!\left[ #1 \right]\!\right]}
\newcommand{\floor}[1]{\left\lfloor #1 \right\rfloor}
\newcommand{\ceil}[1]{\left\lceil #1 \right\rceil}

\DeclareMathOperator{\occ}{occ}
\DeclareMathOperator{\len}{len}

\begin{document}

\maketitle

\begin{abstract}
One can consider $\mu$-Martin-L\"of randomness for a probability measure $\mu$ on $2^{\omega}$, such as the Bernoulli measure $\mu_p$ given $p \in (0, 1)$. We study Bernoulli randomness of sequences in $n^{\omega}$ with parameters $p_0, p_1, \dotsc, p_{n-1}$, and we introduce a biased version of normality. We prove that every Bernoulli random real is normal in the biased sense, and this has the corollary that the set of biased normal reals has full Bernoulli measure in $n^{\omega}$. We give an algorithm for computing biased normal sequences from normal sequences, so that we can give explicit examples of biased normal reals. We investigate an application of randomness to iterated function systems. Finally, we list a few further questions relating to Bernoulli randomness and biased normality.
\end{abstract}

\section{Background}
Mirroring the historical development of normal numbers and algorithmic randomness, this paper introduces some generalizations of known connections between normality and randomness. Borel \cite{borel} first described normal numbers in 1909, and Pillai \cite{pillai} shortened Borel's definition in 1940. One decade later, Niven and Zuckerman \cite{nivenzuckerman} proved an equivalent formulation of normality in terms of blocks of digits. Although Borel also showed in 1909 that almost all real numbers are normal in every base, where the measure is the Lebesgue measure, the first explicit construction of a normal number did not appear until 1933, by Champernowne \cite{champernowne}. In 1966, Martin-L\"of \cite{martinlof} defined randomness criteria in terms of geometrically shrinking and uniformly computably enumerable open sets, and it can be shown that, to the Lebesgue measure, all Martin-L\"of-random numbers are normal in every base.

After introducing preliminary notation, definitions, and theorems in the remainder of this section, we begin in Section \ref{sec:generalizations} with a description of normality with respect to given biases on each digit in the base. This definition is written to follow Borel's original definition of normality. We then prove a redundancy in our definition, as Pillai showed in Borel's definition. We follow this with a characterization of biased normality in terms of blocks, as Niven and Zuckerman proved. Using the terms and definitions that will be introduced later in this paper, the equivalences allow us to prove that, fixing $b$ biases $\bar{p} = (p_0, p_1, \dotsc, p_{b-1})$ adding up to $1$ and using the Bernoulli measure $\mu_{\bar{p}}$ on $b^{\omega}$, all $\mu_{\bar{p}}$-Martin-L\"of-random numbers are biased normal with respect to $\bar{p}$. In Section \ref{sec:construction}, we give an algorithm which, given rational biases, uses a normal number to construct a biased normal number with respect to the biases. Section \ref{sec:applications} describes an application of biased normal numbers to iterated function systems, and Section \ref{sec:questions} lists questions for further research.

\subsection{Notation}
A \textit{base} is an integer $n \geq 2$. Let $n^{\omega}$ denote the set of infinite $n$-ary sequences where $n$ is a base. We identify $n^{< \omega}$ as the set of finite $n$-ary sequences, which we also call \textit{blocks}. For a given $\ell \in \N$, let $n^{\ell}$ be the set of $n$-ary sequences of length $\ell$. If $\sigma \in n^{< \omega}$, then let $\dbracket{\sigma} \subseteq n^{\omega}$ be the set of infinite sequences which extend $\sigma$.

If $\sigma$ is a (finite or infinite) $n$-ary sequence, we will index the entries in $\sigma$ by $\sigma[i]$, where $\sigma[0]$ is the first entry of the sequence. The subsequence of $\sigma$ from index $i$ to index $j$, inclusive, is $\sigma[i : j]$. If $\sigma$ is finite, then the length of $\sigma$ is $\len(\sigma)$. If $\sigma_1, \sigma_2 \in n^{<\omega}$, then $\sigma_1 \sigma_2$ is the concatenation of $\sigma_1$ and $\sigma_2$. The number of occurrences of a base $n$ block $\rho$ inside $\sigma$ is $\occ(\sigma, \rho)$. The empty sequence is denoted as $\epsilon$.

The base $b$ representation of a real number $r \in [0, 1]$ is denoted $(r)_b$ and refers to the sequence in $b^{\omega}$ such that $r = \sum_{i=1}^{\infty} ((r)_b[i - 1] \times b^{-i})$ and such that $(r)_b$ includes infinitely many instances of digits which are not $b - 1$. Occasionally, we will use a sequence in place of its corresponding real number.

\subsection{Probability Measures}

\begin{definition}
	A \textit{Borel probability measure} on $n^{\omega}$ is a countably additive, monotone function $\mu: \mathcal{F} \to [0, 1]$, where $\mathcal{F}$ is the Borel $\sigma$-algebra of $n^{\omega}$ and $\mu(n^{\omega}) = 1$. Since a Borel probability measure is uniquely determined by the values it takes on finite unions of basic open cylinders, when giving a Borel probability measure it is sufficient to specify a function $\rho: n^{< \omega} \to [0, 1]$ satisfying $\rho(\epsilon) = 1$, where $\epsilon$ is the empty sequence, and
	\begin{equation*}
	\rho(\sigma) = \sum_{i=0}^{n-1} \rho(\sigma i)
	\end{equation*}
	where $\sigma i$ denotes the concatenation of $\sigma$ with $i$ as a symbol in base $n$. The resulting measure sets $\mu(\dbracket{\sigma}) = \rho(\sigma)$. For this paper, we will refer to Borel probability measures as \textit{measures} and only identify the underlying function on blocks, so that $\mu(\dbracket{\sigma})$ is written as $\mu(\sigma)$.
\end{definition}

\begin{definition}
	The \textit{Lebesgue measure} $\lambda$ on $n^{\omega}$ is the measure given by setting
	\begin{equation*}
	\lambda(\sigma) = \frac{1}{n^{\len(\sigma)}}
	\end{equation*}
	for each $\sigma \in n^{< \omega}$.
\end{definition}

\begin{definition}
	The \textit{Bernoulli measure} $\mu_{\bar{p}}$ on $n^{\omega}$, with associated positive probabilities $\bar{p} = (p_0, p_1, \dotsc, p_{n - 1})$ satisfying $\sum_{i=0}^{n - 1} p_i = 1$, is the measure given by setting
	\begin{equation*}
	\mu_{\bar{p}}(\sigma) = p_{\sigma[0]} p_{\sigma[1]} \dotsb p_{\sigma[\len(\sigma) - 1]}
	\end{equation*}
	for each $\sigma \in n^{< \omega}$. Note that the Lebesgue measure on $n^{\omega}$ is exactly the Bernoulli measure on $n^{\omega}$ obtained by setting $p_i = \frac{1}{n}$ for each $i$.
\end{definition}

\subsection{Randomness}

\begin{definition}[Martin-L\"of \cite{martinlof}, see also \cite{nies}]
	Let $\mu$ be a measure on $n^{\omega}$ and $z \in n^{\omega}$. A \textit{$\mu$-Martin-L\"of test relative to $z$} is a uniformly computably enumerable (relative to $z$) sequence $(\mathcal{U}_i)_{i \in \omega}$ of subsets of $n^{\omega}$ with $\mu(\mathcal{U}_i) \leq 2^{-i}$ for every $i \in \N$. Say $x \in n^{\omega}$ \textit{passes} the test if $x \not \in \bigcap_{i \in \omega} \mathcal{U}_i$. If $x$ passes every $\mu$-Martin-L\"of test relative to $z$, then $x$ is \textit{$\mu$-Martin-L\"of random relative to $z$}.
\end{definition}

\begin{definition}
	If $x \in n^{\omega}$ is $\mu_{\bar{p}}$-Martin-L\"of random for the Bernoulli measure $\mu_{\bar{p}}$ with some probabilities $\bar{p} = (p_0, p_1, \dotsc, p_{n-1})$, then $x$ is \textit{Bernoulli random} with respect to the parameters $\bar{p}$.
\end{definition}

Bernoulli randomness for binary sequences has been studied by Porter in \cite{porter}.

\subsection{Fragments of Randomness}
\begin{definition}
	A real number $x$ is \textit{simply normal} to base $b$ if every base $b$ digit $d \in \{0, 1, \dotsc, b - 1\}$ appears with density $\frac{1}{b}$ in $(x)_b$. That is,
	\begin{equation*}
		\lim_{n \to \infty} \frac{\occ((x)_b [0 : n - 1], d)}{n} = \frac{1}{b}
	\end{equation*}
\end{definition}

Borel characterized normality in the following way.

\begin{definition}[Borel \cite{borel}]\label{def:borelnormal}
	A real number $x$ is \textit{normal} to base $b$ if for every natural $n$ and positive integer $k$, $b^n x$ is simply normal to base $b^k$.
\end{definition}

\begin{example}
	In 1933, Champernowne \cite{champernowne} gave an explicit real number which is normal to base 10. 
	\begin{equation*}
		C_{10} = 0.12345678910111213\dotsc
	\end{equation*}
	In general, let $C_n$ denote the real number with the base $n$ representation obtained by concatenating the base $n$ numbers in order. $C_n$ is normal to base $n$.
\end{example}

\begin{example}
	Among the results by Copeland and Erd\H{o}s in \cite{copelanderdos} is the fact that the real number $CE_{n}$ obtained by concatenating the primes in base $n$ in order is normal to base $n$. Then
	\begin{align*}
		CE_{10} &= 0.2357111317192329\dotsc\\
		CE_{3} &=  0.2101221102111122\dotsc
	\end{align*}
\end{example}

In 1940, Pillai simplified Borel's definition with the following theorem.

\begin{theorem}[Pillai \cite{pillai}]\label{thm:pillai}
	A real number $x$ is normal to base $b$ if and only if for every positive integer $k$, $x$ is simply normal to base $b^k$.
\end{theorem}

In 1950, another equivalence was proven by Niven and Zuckerman.

\begin{theorem}[Niven and Zuckerman \cite{nivenzuckerman}]\label{thm:nivenzuckerman}
	A real number $x$ is normal to base $b$ if and only if for every positive integer $\ell$, every block $w \in b^{\ell}$ appears in $(x)_b$ with frequency $\frac{1}{b^{\ell}}$.
	\begin{equation*}
		\lim_{n \to \infty} \frac{\occ((x)_b [0 : n - 1], w)}{n} = \frac{1}{b^{\ell}}
	\end{equation*}
\end{theorem}

One important connection between normal numbers and algorithmic randomness is the following theorem.

\begin{theorem}\label{thm:mlrandomabsolutely}
	Every $\lambda$-Martin-L\"of random real is \textit{absolutely normal} --- normal in every base.
\end{theorem}

\section{Generalizations}\label{sec:generalizations}

The goal of this section is to prove a version of Theorem \ref{thm:mlrandomabsolutely} for Bernoulli random numbers. To do this, we define a notion of normality given biases on the digits. We will mirror the historical development of normality by generalizing Borel's original definitions of \textit{simply normal} and \textit{normal} to allow for given biases on the digits. In base $b$, the biases $p_0, p_1, \dotsc, p_{b-1}$, also called ``densities'' or ``probabilities'', will be assumed to be positive real numbers adding to $1$.

\begin{definition}
	A real number $x$ is \textit{biased simply normal} to the biases $p_0, p_1, \dotsc, p_{b-1}$ if each base $b$ digit $d \in \{0, 1, \dotsc, b-1\}$ appears with density $p_d$ in $(x)_b$. That is,
	\begin{equation*}
		\lim_{n \to \infty} \frac{\occ((x)_b [0 : n - 1], d)}{n} = p_d
	\end{equation*}
\end{definition}

\begin{definition}\label{def:biasednormal}
	A real number $x$ is \textit{biased normal} with respect to the biases $p_0, p_1, \dotsc, p_{b - 1}$ if for every natural $n$ and positive integer $k$, $b^n x$ is biased simply normal to $p_{k, 0}^*, p_{k, 1}^*, \dotsc, p_{k, b^k - 1}^*$, where for each $i \in \{0, \dotsc, b^k - 1\}$,
	\begin{equation*}
		p_{k, i}^* = \prod_{j=0}^{k - 1} p_{(i)_{b}[j]}
	\end{equation*}
	and where here $(i)_{b}$ contains sufficient zero-padding so that it has exactly $k$ digits.
	
	The frequencies $\bar{p^*} = (p_{k, 0}^*, p_{k, 1}^*, \dotsc, p_{k, b^k - 1}^*)$ are such that if $w^*$ is a base $b^k$ block and $w$ is the corresponding base $b$ block, then $\mu_{\bar{p^*}}(w^*) = \mu_{\bar{p}}(w)$.
\end{definition}

As shown for the case of normality in Theorems \ref{thm:pillai} and \ref{thm:nivenzuckerman}, the definition of \textit{biased normal} can be simplified. To prove this, we will require the following definition.

\begin{definition}\label{def:simplediscrepancy}
	Let $w$ be a length $\ell$ block of digits in base $b$. Let $p_0, p_1, \dotsc, p_{b-1}$ be biases. Then the \textit{simple discrepancy} of $w$ with respect to the biases is
	\begin{equation*}
		\max_{d \in \{0, 1, \dotsc, b - 1\}} \abs{\frac{\occ(w, d)}{\ell} - p_d}
	\end{equation*}
\end{definition}

\begin{lemma}\label{lem:multinom}
	Fix a base $b$, a digit $d$, and a block length $k$. Let $S_i \subseteq b^k$ be the set of blocks of length $k$ containing exactly $i$ instances of $d$. The Bernoulli measure of $S_i$ is
	\begin{equation*}
	\mu_{\bar{p}}(S_i) = \binom{k}{i} p_d^i (1 - p_d)^{k - i}
	\end{equation*}
\end{lemma}
\begin{proof}
	We know that the number of blocks in $S_i$ is
	\begin{equation*}
		\card{S_i} = \binom{k}{i} (b - 1)^{k - i}
	\end{equation*}
	since there are $\binom{k}{i}$ choices for where to put the $i$ instances of $d$ and $k - i$ places where one of $b - 1$ digits occur. We assume without loss of generality and for ease of notation that $d = 0$. For $w \in S_i$, let
	\begin{equation*}
		n_e = \occ(w, e)
	\end{equation*}
	for a digit $e$ in base $b$. The measure of any such $w$ is
	\begin{equation*}
		\mu_{\bar{p}}(w) = p_0^i \prod_{m = 1}^{b - 1} p_m^{n_m}
	\end{equation*}
	To find the measure of $S_i$, we can take the sum of the measures over all such $w$ with digit counts $n_1, n_2, \dotsc, n_{b - 1} \in \N$ such that $\sum_{m = 1}^{b - 1} n_m = k - i$. The number of such $w$ is
	\begin{equation*}
		\sum_{n_1 + n_2 + \dotsb + n_{b - i} = k - i} \binom{k}{i} \binom{k - i}{n_1, n_2, \dotsc, n_{b - 1}}
	\end{equation*}
	where
	\begin{equation*}
		\binom{k - i}{n_1, n_2, \dotsc, n_{b - 1}} = \frac{(k - i)!}{n_1 ! n_2 ! \dotsb n_{b - 1}!}
	\end{equation*}
	is the multinomial coefficient. This is because there are $\binom{k}{i}$ many choices for the locations of $d = 0$, and for each sum $n_1 + n_2 + \dotsb + n_{b - 1} = k - i$ there are $\binom{k - i}{n_1, n_2, \dotsc, n_{b - 1}}$ different length $k - i$ sequences $w$ with $\occ(w, e) = n_e$ for each $e$ from $1$ to $b - 1$. So the measure of $S_i$ is
	\begin{align*}
		\mu_{\bar{p}}(S_i) &= \sum_{n_1 + n_2 + \dotsb + n_{b - i} = k - i} \binom{k}{i} \binom{k - i}{n_1, n_2, \dotsc, n_{b - i}} p_0^i \prod_{j = 1}^{b - 1} p_j^{n_j}\\
		\mu_{\bar{p}}(S_i) &= \binom{k}{i} p_0^i \sum_{n_1 + n_2 + \dotsb + n_{b - i} = k - i} \binom{k - i}{n_1, n_2, \dotsc, n_{b - i}} \prod_{j = 1}^{b - 1} p_j^{n_j}
	\end{align*}
	By the multinomial theorem \cite{feller},
	\begin{equation*}
		\sum_{n_1 + n_2 + \dotsb + n_{b - i} = k - i} \binom{k - i}{n_1, n_2, \dotsc, n_{b - i}} \prod_{j = 1}^{b - 1} p_j^{n_j} = \left(\sum_{j = 1}^{b - 1} p_j\right)^{k - i}
	\end{equation*}
	Therefore
	\begin{align*}
		\mu_{\bar{p}}(S_i) &= \binom{k}{i} p_0^i \left(\sum_{j = 1}^{b - 1} p_j\right)^{k - i}
		\intertext{and we know $\sum_{j = 1}^{b - 1} p_j = 1 - p_0$, so}
		\mu_{\bar{p}}(S_i) &= \binom{k}{i} p_0^i \left(1 - p_0\right)^{k - i}
	\end{align*}
	which is the desired equality for $d = 0$.
\end{proof}

\begin{lemma}\label{lem:badblocks}
	Let $0 < \varepsilon < \min(p_0, \dotsc, p_{b-1})$. Fix a block length $k$. Say that a block $w$ of length $k$ is ``bad'' for a digit $d$ if
	\begin{align*}
	\occ(w, d) &\leq (p_d - \varepsilon)k
	\intertext{or}
	\occ(w, d) &\geq (p_d + \varepsilon)k
	\end{align*}
	Let $B$ be the set of such $w$.
	\begin{equation*}
	B = \{w \in b^k : \abs{\occ(w, d) - p_d} \geq \varepsilon k\}
	\end{equation*}
	Then the Bernoulli measure of $B$ in $b^{\omega}$ with parameters $p_0, p_1, \dotsc, p_{b - 1}$ is at most $2e^{-2 \varepsilon^2 k}$.
\end{lemma}
\begin{proof}
	Let $i$ be an integer such that $0 \leq i \leq k$. Let $B_i$ be set of blocks of length $k$ containing exactly $i$ instances of the digit $d$. The Bernoulli measure of $B_i$ in $b^{\omega}$ with parameters $p_0, p_1, \dotsc, p_{b - 1}$ is, by Lemma \ref{lem:multinom},
	\begin{equation*}
	\mu_{\bar{p}}(B_i) = \binom{k}{i} (p_d)^i (1 - p_d)^{k - i}
	\end{equation*}
	Notice that this is the binomial distribution with $k$ trials and $i$ successes, where the probability of success is $p_d$. To calculate $\mu_{\bar{p}}(B)$, we have
	\begin{align*}
	B &=\ \ \bigcup_{i=0}^{\mathclap{\floor{(p_d - \varepsilon)k}}} B_i\ \ \cup\ \ \ \ \bigcup_{\mathclap{i=\ceil{(p_d + \varepsilon)k}}}^{k} B_i
	\intertext{where all the unions are of pairwise disjoint sets. Then}
	\mu_{\bar{p}}(B) &= \sum_{i=0}^{\mathclap{\floor{(p_d - \varepsilon)k}}} \mu_{\bar{p}}(B_i)\ +\ \sum_{\mathclap{i=\ceil{(p_d + \varepsilon)k}}}^{k} \mu_{\bar{p}}(B_i)
	\end{align*}
	We expand both appearances of $\mu_{\bar{p}}(B_i)$ as above.
	\begin{equation*}
	\mu_{\bar{p}}(B) = \sum_{i=0}^{\floor{(p_d - \varepsilon)k}} \binom{k}{i} (p_d)^i (1 - p_d)^{k - i} + \sum_{i=\ceil{(p_d + \varepsilon)k}}^{k} \binom{k}{i} (p_d)^i (1 - p_d)^{k - i}
	\end{equation*}
	Apply Hoeffding's inequality \cite{vershynin} on the tail ends of the binomial distribution to get that
	\begin{align*}
	\sum_{i=0}^{\floor{(p_d - \varepsilon)k}} \binom{k}{i} (p_d)^i (1 - p_d)^{k - i} &\leq e^{-2\varepsilon^2 k}
	\intertext{and}
	\sum_{i=\ceil{(p_d + \varepsilon)k}}^{k} \binom{k}{i} (p_d)^i (1 - p_d)^{k - i} &\leq e^{-2\varepsilon^2 k}
	\end{align*}
	It follows that $\mu_{\bar{p}}(B) \leq 2e^{-2\varepsilon^2 k}$.
\end{proof}

Definition \ref{def:borelnormal}, Theorem \ref{thm:pillai}, and Theorem \ref{thm:nivenzuckerman} give three equivalent characterizations of normality. The next three lemmas accomplish the same task for biased normality.

\begin{lemma}\label{lem:biasednormalityalt0}
	If $x$ is biased normal to $p_0, p_1, \dotsc, p_{b-1}$, then for every positive integer $k$, $x$ is biased simply normal to $p_{k, 0}^*, p_{k, 1}^*, \dotsc, p_{k, b^k - 1}^*$, where for each $i \in \{0, \dotsc, b^k - 1\}$,
	\begin{equation*}
		p_{k, i}^* = \prod_{j=0}^{k - 1} p_{(i)_{b}[j]}
	\end{equation*}
\end{lemma}
\begin{proof}
	This lemma follows immediately from the definition of \textit{biased normal}, as it is a special case of the definition.
\end{proof}

\begin{lemma}\label{lem:biasednormalityalt1}
	If for every positive integer $k$, $x$ is biased simply normal to $p_{k, 0}^*, p_{k, 1}^*, \dotsc, p_{k, b^k - 1}^*$, where for each $i \in \{0, \dotsc, b^k - 1\}$,
	\begin{equation*}
		p_{k, i}^* = \prod_{j=0}^{k - 1} p_{(i)_{b}[j]}
	\end{equation*}
	then for each positive integer $r$ and each block $v \in b^r$,
	\begin{equation*}
		\lim_{n \to \infty} \frac{\occ((x)_b[0 : n - 1], v)}{n} = \prod_{j=0}^{r - 1} p_{v[j]} = \mu_{\bar{p}}(v)
	\end{equation*}
\end{lemma}

\begin{proof}
	Fix $r$ and $v \in b^r$. Let $\varepsilon_1, \varepsilon_2 > 0$. By Lemma $\ref{lem:badblocks}$, there is a sufficiently large positive integer $N_0$ such that all $N \geq N_0$, all but a $\mu_{\bar{p}}$-measure at most $\varepsilon_1$ subset $B_0$ of length $N$ base $b$ blocks have simple discrepancy less than $\varepsilon_2$ when parsed in length $r$ intervals starting from index $0$. Moreover, we argue that $N$ can be made sufficiently large so that for each $m$ from $0$ to $r - 1$, all but a $\mu_{\bar{p}}$-measure $b^m \varepsilon_1$ subset $B_m$ of length $N$ base $b$ blocks have simple discrepancy less than $\varepsilon_2$ when parsed in length $r$ intervals starting from index $m$. The $\mu_{\bar{p}}$-measure of each $B_m$ is at most $b^m \varepsilon_1$ because each length $N - m$ sequence extends to a length $N$ sequence in $b^m$ many ways, and we know $\mu_{\bar{p}}(B_0) \leq \varepsilon_1$. Thus the measure of $\bigcup_{m=0}^{r - 1} B_m$ is at most $\sum_{m=0}^{r - 1} b^m \varepsilon_1 \leq b^r \varepsilon_1$.
	
	We compute an upper bound on the eventual frequency of $v$ in $(x)_b$. Let $\varepsilon > 0$. Parse $(x)_b$ in length $N$ subblocks starting from index $0$, where $N$ will be sufficiently large as will be determined by the following analysis. Because $x$ is biased simply normal in base $b^N$, there is a positive integer $\ell_0$ such that for all $\ell \geq \ell_0$, every $w \in b^N$ occurs within $\varepsilon$ of its expected frequency in the first $\ell$ digits of $(x)_{b^N}$. That is,
	\begin{equation*}
		\abs{\frac{\occ((x)_{b^N}[0 : \ell - 1], w)}{\ell} - \mu_{\bar{p_N^*}}(w)} \leq \varepsilon
	\end{equation*}
	for every $w \in b^N$, where $\bar{p_N^*} = (p_{N, 0}^*, \dotsc, p_{N, b^N - 1}^*)$. Parsing $(x)_{b}$ in length $N$ blocks, instances of $v$ in $(x)_{b}$ can occur in three different ways. If an instance of $v$ is not contained within a length $N$ block when parsing $(x)_b$ into length $N$ subblocks starting from index $0$, then $v$ begins in one block and ends in the next block. All other instances of $v$ will be entirely within one length $N$ subblock, and we say such a block $w$ is ``good'' if $\abs{\frac{\occ(w, v)}{N} - \mu_{\bar{p}}(v)} \leq \varepsilon$, or ``bad'' otherwise. If an instance of $v$ is contained in a length $N$ block $w$, then we consider separately the cases that the block is good or bad.
	
	Let $\varepsilon > 0$. There are $\frac{\ell (r - 1)}{N}$ many length $r$ blocks that start in one length $N$ block and end in another length $N$ block. Some of those $\frac{\ell (r - 1)}{N}$ blocks could be instances of $v$, and none of them are counted in the above computation. Assume that all $\frac{\ell (r - 1)}{N}$ of these blocks are instances of $v$. Since $N$ is made arbitrarily large, $\frac{\ell (r - 1)}{N} < \varepsilon \ell$.
	
	Next, we bound the occurrences of $v$ in bad length $N$ subblocks. By Lemma \ref{lem:badblocks}, the subset $B$ of bad length $N$ blocks has $\mu_{\bar{p}}$-measure at most $2e^{-2\varepsilon^2 N}$. Since $N$ is made arbitrarily large, we can assume $2e^{-2\varepsilon^2 N} \leq \varepsilon$. Assume every bad length $N$ block has $N - r + 1$ occurrences of $v$, the maximum possible number of occurrences. By the choice of $\ell$, the number of digits in $(x)_{b^N}[0 : \ell - 1]$ which are bad base $b$ length $N$ blocks is at most $\varepsilon \ell$. We are assuming each of these bad blocks contains $N - r + 1$ instances of $v$, so the number of instances of $v$ in bad blocks is at most $\varepsilon (N - r + 1)\ell$.
	
	Similarly, let $G$ be the set of length $N$ good blocks. There are at most $\ell$ many elements of $G$ among the digits of $(x)_{b^N}[0 : \ell - 1]$. In a good block, the frequency of $v$ is within $\varepsilon$ of its expected frequency. The number of instances of $v$ in good blocks is at most $\ell(\mu_{\bar{p}}(v) + \varepsilon)(N - r + 1)$.
	
	We have counted the instances of $v$ in $(x)_b[0 : N \ell - 1]$ between two length $N$ blocks, inside bad blocks, and inside good blocks. Now we can compute an upper bound on the frequency of $v$ in the first $N \ell$ digits of $(x)_b$. We have
	\begin{gather*}
		\frac{\occ((x)_b[0 : N \ell - 1], v)}{N \ell} \leq \frac{\varepsilon \ell + \varepsilon(N - r + 1)\ell + \ell(\mu_{\bar{p}}(v) + \varepsilon)(N - r + 1)}{N \ell}
		\intertext{by above. Additionally,}
		\frac{\varepsilon \ell + \varepsilon(N - r + 1)\ell + \ell(\mu_{\bar{p}}(v) + \varepsilon)(N - r + 1)}{N \ell} = \frac{\varepsilon + \varepsilon(N - r + 1) + (\mu_{\bar{p}}(v) + \varepsilon)(N - r + 1)}{N}
		\intertext{and since $N - r + 1 \leq N$,}
		\frac{\varepsilon + \varepsilon(N - r + 1) + (\mu_{\bar{p}}(v) + \varepsilon)(N - r + 1)}{N} \leq \frac{\varepsilon + \varepsilon N + (\mu_{\bar{p}}(v) + \varepsilon) N}{N} = \frac{\varepsilon}{N} + 2\varepsilon + \mu_{\bar{p}}(v).
		\intertext{Therefore}
		\frac{\occ((x)_b[0 : N \ell - 1])}{N \ell} \leq \frac{\varepsilon}{N} + 2\varepsilon + \mu_{\bar{p}}(v)
	\end{gather*}
	which approaches $\mu_{\bar{p}}(v)$ as required. The computation for a lower bound on the eventual frequency of $v$ in $(x)_{b}$ can be made in a way analogous to the computation above; again parsing $(x)_b$ in length $N$ subblocks, assume that all occurrences of $v$ are within good length $N$ blocks. By Lemma \ref{lem:badblocks}, there are at least $(1 - \varepsilon)\ell$ many good length $N$ blocks when $\ell$ is sufficiently large. Each good length $N$ block must contain at least $(N - r + 1)(\mu_{\bar{p}}(v) - \varepsilon)$ instances of $v$. Then the number of occurrences of $v$ is at least $(1 - \varepsilon) \ell (N - r + 1)(\mu_{\bar{p}}(v) - \varepsilon)$, and one can check that the frequency of $v$ in $(x)_b[0 : N \ell - 1]$ again approaches $\mu_{\bar{p}}(v)$ as required.
\end{proof}

\begin{lemma}\label{lem:biasednormalityalt2}
	If $x$ is such that for every positive integer $r$ and every block $v \in b^r$,
	\begin{equation*}
		\lim_{n \to \infty} \frac{\occ((x)_b[0 : n - 1], v)}{n} = \prod_{j=0}^{r - 1} p_{v[j]} = \mu_{\bar{p}}(v)
	\end{equation*}
	then $x$ is biased normal as in Definition \ref{def:biasednormal}.
\end{lemma}
\begin{proof}
	This proof is similar to a proof by Cassels in \cite{cassels} for the case of normality, and we use similar notation. Let $f$ and $g$ be base $b$ blocks of lengths $r$ and $s$ respectively, with $s \geq r$. For a given integer $m$ from $0$ to $r - 1$, $R_m(g, f)$ is the number of solutions to $g[n : n + r - 1] = f$ with $n \equiv m\ (\mathrm{mod}\ r)$. Then $R_m(g, f) \leq s - r + 1$.
	
	Let $\varepsilon > 0$ and fix a block $v$ in base $b$ of length $r$. Let $s \geq r$ be a positive integer. Consider $v$ as a digit in base $b^r$. Let $B$ be the set of length $s$ base $b$ blocks with simple discrepancy at least $\varepsilon$. By Lemma \ref{lem:badblocks}, we have
	\begin{equation}\label{eqn:cassels1}
		\max_{0 \leq m < r} \abs{R_m(w, v) - \frac{(s - r + 1)\mu_{\bar{p}}(v)}{r}} < \varepsilon (s - r + 1)
	\end{equation}
	for all $w \in b^s$, except for a subset $B \subseteq b^s$ of length $s$ blocks which has Bernoulli measure at most $2e^{-2\varepsilon^2 s}$. For sufficiently large $s$, the Bernoulli measure of $B$ is less than $\varepsilon$. Because $(x)_b$ has the expected frequency of occurrences of length $s$ blocks, there exists $N$ such that the number of occurrences of blocks from $B$ in the first $N - s + 1$ digits in $x$ is at most $2 \varepsilon N$. For each $m$, let
	\begin{align*}
		i_m &= (s - r + 1) R_m((x)_b [0 : N - 1], v)\\
		j_m &= \sum_{t=0}^{N - s} R_{m-t}((x)_b [t : t + s - 1], v)
	\end{align*}
	Each occurrence of $v$ in $x$ at a starting index $n \equiv m\ \mathrm{mod}\ r$ contributes $s - r + 1$ to $i_m$. The same holds for $j_m$, except for occurrences of $v$ which start in $x$ at an index from $0$ to $s - 2$ or from $N - s - 3$ to $N - 1$, which contribute less than $s - r + 1$ to $j_m$. Then for each $m$, $\abs{i_m - j_m} \leq 2(s - 1)(s - r + 1) \leq 2s^2$.
	
	Each of the $2 \varepsilon N$ blocks appearing in $(x)_{b}[0 : N - 1]$ from $B$ contribute at most $s - r + 1$ occurrences of $v$. For length $s$ blocks appearing in $(x)_{b}[0 : N - 1]$ which are not members of $B$, $v$ appears at starting indices equivalent to $m\ \mathrm{mod}\ r$ with frequency at most $\frac{\mu_{\bar{p}}(v) + \varepsilon}{r}$ by equation \ref{eqn:cassels1}, so the number of these occurrences of $v$ in such length $s$ blocks is at most $\frac{(\mu_{\bar{p}}(v) + \varepsilon)(s - r + 1)}{r}$. There are at most $N - s + 1$ length $s$ blocks. This gives the upper bound 
	\begin{equation*}
		j_m \leq 2 \varepsilon N (s - r + 1) + (N - s + 1) (\mu_{\bar{p}}(v) + \varepsilon)(s - r + 1)
	\end{equation*}
	for each $m$. Then an upper bound on $\frac{j_m}{s - r + 1}$ is
	\begin{equation*}
		\frac{j_m}{s - r + 1} \leq 2 \varepsilon N + \frac{(N - s + 1)\mu_{\bar{p}}(v)}{r} + \varepsilon(N - s + 1)
	\end{equation*}
	for each $m$, where, to match the bounds given by Cassels, we have used the fact that $\frac{\varepsilon}{r} \leq \varepsilon$. Note that
	\begin{gather*}
		\abs{\frac{i_m}{s - r + 1} - \frac{j_m}{s - r + 1}} \leq \frac{2s^2}{s - r + 1}
		\intertext{and}
		\frac{i_m}{s - r + 1} = R_m((x)_b [0 : N - 1], v)
		\intertext{since $\abs{i_m - j_m} \leq 2s^2$ and by definition of $i_m$. Thus}
		\abs{R_m((x)_b [0 : N - 1], v) - \frac{(N - s + 1) \mu_{\bar{p}(v)}}{r}} \leq \frac{2s^2}{s - r + 1} + \varepsilon(N - s + 1) + 2 \varepsilon N
	\end{gather*}
	and
	\begin{equation*}
		\limsup_{N \to \infty} \abs{\frac{R_m ((x)_b [0 : N - 1], v)}{N} - \frac{\mu_{\bar{p}}(v)}{r}} \leq 3 \varepsilon.
	\end{equation*}
	Since $\varepsilon$ is arbitrarily small, we therefore have
	\begin{equation*}
		\lim_{N \to \infty} \frac{R_m((x)_b [0 : N - 1], v)}{N} = \frac{\mu_{\bar{p}}(v)}{r}
	\end{equation*}
	for each $m$ from $0$ to $r - 1$. Conclude that $x$ is biased normal as in Definition \ref{def:biasednormal}.
\end{proof}

Together, Lemmas \ref{lem:biasednormalityalt0}, \ref{lem:biasednormalityalt1}, and \ref{lem:biasednormalityalt2} prove the following corollary.

\begin{corollary}\label{cor:equivdefsbiased}
	Let $x$ be a real number. Fix a base $b$ and densities $p_0, \dotsc, p_{b-1}$. The following are equivalent.
	\begin{enumerate}[(1)]
		\item $x$ is biased normal as in Definition \ref{def:biasednormal}.
		\item For every positive integer $k$, $x$ is biased simply normal to $p_{k, 0}^*, p_{k, 1}^*, \dotsc, p_{k, b^k - 1}^*$, where for each $i \in \{0, \dotsc, b^k - 1\}$,
		\begin{equation*}
			p_{k, i}^* = \prod_{j=0}^{k - 1} p_{(i)_{b}[j]}
		\end{equation*}
		\item For each positive integer $r$ and for each $v \in b^{r}$,
		\begin{equation*}
		\lim_{n \to \infty} \frac{\occ((x)_b[0 : n - 1], v)}{n} = \prod_{j=0}^{r - 1} p_{v[j]} = \mu_{\bar{p}}(v)
		\end{equation*}
	\end{enumerate}
\end{corollary}

\begin{theorem}\label{thm:bernoullinormal}
	Let $x$ be a Bernoulli random real, with biases $p_0, p_1, \dotsc, p_{n - 1}$. Then $x$ is biased normal with respect to $p_0, p_1, \dotsc, p_{n - 1}$.
\end{theorem}
\begin{proof}
	We will construct a $\mu_{\bar{p}}$-Martin-L\"of test. Let $0 < \varepsilon < \min(p_0, \dotsc, p_{b-1})$. Let $k_0$ be the least such that Lemma \ref{lem:badblocks} holds for $\varepsilon$ and $b$. For each integer $k \geq k_0$, let 
	\begin{equation*}
		B_k = \bigcup_{N > k} \{w \in b^N : \text{$\abs{\occ(w, d) - p_d} > \varepsilon N$ for some digit $d$ in base $b$}\}		
	\end{equation*}
	Then
	\begin{equation*}
		\mu_{\bar{p}}(B_k) \leq \sum_{N > k} 2e^{-2 \varepsilon^2 N} \leq \int_{k}^{\infty} 2 e^{-2\varepsilon^2 N} dN = \frac{e^{-2 \varepsilon^2 k}}{\varepsilon^2}
	\end{equation*}
	Suppose $x$ is not biased normal to the densities $p_0, \dotsc, p_{b-1}$. By Corollary \ref{cor:equivdefsbiased}, $x$ is equivalently not biased simply normal to base $b^n$ for some positive integer $n$ and densities $p_{n, 0}^*, \dotsc, p_{n, b^n - 1}^*$ as defined in Corollary \ref{cor:equivdefsbiased}. Then $x \in \bigcap_{k \geq k_0} B_k$, and $x$ fails the $\mu_{\bar{p}}$-Martin-L\"of-random test.
\end{proof}

\begin{corollary}
	Fixing densities $p_0, p_1, \dotsc, p_{b-1}$, the set of biased normal reals has Bernoulli measure $1$.
\end{corollary}

As another corollary of Theorem \ref{thm:bernoullinormal}, we can prove Theorem \ref{thm:mlrandomabsolutely}.

\begingroup
\def\thetheorem{\ref{thm:mlrandomabsolutely}}
\begin{theorem}
	Every $\lambda$-Martin-L\"of random real is \textit{absolutely normal} --- normal in every base.
\end{theorem}
\addtocounter{theorem}{-1}
\endgroup

\begin{proof}
	Let $x$ be a $\lambda$-Martin-L\"of-random real. Let $b$ be any base, and let $\bar{p} = (p_0, p_1, \dotsc, p_{b-1})$ where $p_i = \frac{1}{b}$ for all $i$. Because the Bernoulli measure with parameters $\bar{p}$ is the Lebesgue measure, and $x$ is $\lambda$-Martin-L\"of-random, it follows that $x$ is Bernoulli random with parameters $\bar{p}$. By Theorem \ref{thm:bernoullinormal}, $x$ is biased normal with respect to $\bar{p}$. The parameters $\bar{p}$ are uniform, so equivalently, $x$ is normal to base $b$. Since $b$ was arbitrary, deduce that $x$ is absolutely normal.
\end{proof}

\section{Construction of Biased Normal Sequences}\label{sec:construction}
We present a simple algorithm for computing a biased normal sequence by using a normal sequence, but we must assume that the given probabilities are rational numbers.

\begin{construction}\label{def:biasednormalcomp}
	Let $p_0, p_1, p_2, \dotsc, p_{n-1}$ be positive rational probabilities adding up to $1$. For each $i \in \{0, 1, 2, \dotsc, n - 1\}$, let $p_i = \frac{a_i}{b_i}$, with $a_i, b_i$ being positive coprime integers. Let $d = \mathrm{lcm}(b_0, b_1, \dotsc, b_{n-1})$. Then there is a base $n$ block $g$ of length $d$ containing exactly $p_i d$ of each $i$, as $p_i d$ is an integer. Assume $g$ has the base $n$ digits in increasing order. Next, let $\nu \in d^{\omega}$ be base $d$ normal sequence. Construct the sequence $\beta \in n^{\omega}$ from $\nu$ by setting $\beta[k] = g[\nu[k]]$.
\end{construction}

\begin{example}
	Let $p_0 = \frac{2}{3}$ and $p_1 = \frac{1}{3}$. Then $d = 3$, and we can let $g = 001$. This means that for each $k \in \N$, $\beta[k]$ will be $0$ if $\nu[k]$ is $0$ or $1$, and $\beta[k]$ will be $1$ if $\nu[k]$ is $2$. If $\nu$ is Champernowne's base 3 sequence,
	\begin{align*}
	\nu &= 0121011122021221\dotsc
	\intertext{then $\beta$ begins}
	\beta &= 0010000011010110\dotsc
	\end{align*}
\end{example}

\begin{theorem}
	In Construction \ref{def:biasednormalcomp}, $\beta$ is biased normal with respect to $p_0, p_1, p_2, \dotsc, p_{n-1}$.
\end{theorem}
\begin{proof}
	Let $w \in n^{\ell}$. By Corollary \ref{cor:equivdefsbiased}, it is sufficient to show that $w$ has its expected frequency $\mu_{\bar{p}}(w)$ in $\beta$. Let $\nu$ be the base $d$ normal sequence used to construct $\beta$. We will rely on the normality of $\nu$. 
	
	Define $A_w$ to be the set of length $\ell$ blocks $u$ in base $d$ such that $g[u[i]] = w[i]$ for all $i$ from $0$ to $\ell - 1$. In other words, a block $u \in A_w$ appears starting at index $k$ in $\nu$ if and only if $w$ appears starting at index $k$ in $\beta$. The number of blocks in $A_w$ is
	\begin{equation*}
		\card{A_w} = \prod_{i=0}^{\ell - 1} (p_{w[i]} d) = d^{\ell} \prod_{i=0}^{\ell - 1} p_{w[i]} = d^{\ell} \mu_{\bar{p}}(w)
	\end{equation*}
	by construction of $g$. By normality of $\nu$ and Theorem \ref{thm:nivenzuckerman}, every base $d$ block $u$ of length $\ell$ appears with frequency $\frac{1}{d^{\ell}}$ in $\nu$.
	\begin{equation*}
		\lim_{k \to \infty} \frac{\occ(\nu[0 : k - 1], u)}{k} = \frac{1}{d^{\ell}}
	\end{equation*}
	Let $\varepsilon > 0$. Then there exists $k_0 \in \N$ such that for all $k \geq k_0$ and each $u \in d^{\ell}$,
	\begin{equation*}
		\abs{\frac{\occ(\nu[0 : k - 1], u)}{k} - \frac{1}{d^{\ell}}} < \varepsilon
	\end{equation*}
	Consider $k \geq k_0$. For each $u \in d^{\ell}$, let $\delta_u$ be such that $\abs{\delta_u} \leq \varepsilon$ and
	\begin{equation*}
		\frac{\occ(\nu[0 : k - 1], u)}{k} = \frac{1}{d^{\ell}} + \delta_u
	\end{equation*}
	By the construction of $\beta$, we can count instances of $w$ in $\beta$ in terms of instances of $u \in A_w$ appearing in $\nu$.
	\begin{align*}
		\occ(\beta[0 : k - 1], w) &= \sum_{u \in A_w} \occ(\nu[0 : k - 1], u)
		\intertext{Then}
		\frac{\occ(\beta[0 : k - 1], w)}{k} &= \sum_{u \in A_w} \frac{\occ(\nu[0 : k - 1], u)}{k}
		\intertext{and by above,}
		\frac{\occ(\beta[0 : k - 1], w)}{k} &= \sum_{u \in A_w} \left(\frac{1}{d^{\ell}} + \delta_u \right)
	\end{align*}
	Since $\abs{\delta_u} \leq \varepsilon$, we then have
	\begin{gather*}
		\sum_{u \in A_w} \left(\frac{1}{d^{\ell}} - \varepsilon\right) < \frac{\occ(\beta[0 : k - 1], w)}{k} < \sum_{u \in A_w} \left(\frac{1}{d^{\ell}} + \varepsilon\right)
		\intertext{and we calculated $|A_w| = d^{\ell} \mu_{\bar{p}}(w)$, so}
		d^{\ell} \mu_{\bar{p}}(w) \left(\frac{1}{d^{\ell}} - \varepsilon\right) < \frac{\occ(\beta[0 : k - 1], w)}{k} < d^{\ell} \mu_{\bar{p}}(w) \left(\frac{1}{d^{\ell}} + \varepsilon\right)\\
		\mu_{\bar{p}}(w) - \varepsilon d^{\ell} \mu_{\bar{p}}(w) < \frac{\occ(\beta[0 : k - 1], w)}{k} < \mu_{\bar{p}}(w) + \varepsilon d^{\ell} \mu_{\bar{p}}(w)
		\intertext{Thus}
		\abs{\frac{\occ(\beta[0 : k - 1], w)}{k} - \mu_{\bar{p}}(w)} < \varepsilon d^{\ell} \mu_{\bar{p}}(w)
		\intertext{Since $\varepsilon$ is arbitrarily small and $d^{\ell} \mu_{\bar{p}}(w)$ is constant, deduce that}
		\lim_{k \to \infty} \frac{\occ(\beta[0 : k - 1], w)}{k} = \mu_{\bar{p}}(w)
	\end{gather*}
	and that, by Corollary \ref{cor:equivdefsbiased}, $\beta$ is biased normal with respect to the probabilities.
\end{proof}

Because the translation described in Construction \ref{def:biasednormalcomp} is measure-preserving, computable, and continuous, we have the following theorem.

\begin{theorem}
	Let $x$ be a $\lambda$-Martin-L\"of-random real, let $b$ be a base, and let $p_0, p_1, \dotsc, p_{b-1}$ be rational densities. Let $\beta$ be the result of running Construction \ref{def:biasednormalcomp} on $(x)_b$. Then $\beta$ is Bernoulli random with parameters $p_0, p_1, \dotsc, p_{b-1}$.
\end{theorem}

\section{Application: Iterated Function Systems}\label{sec:applications}
In his book \textit{Fractals Everywhere} \cite{barnsley} on the theory of iterated function systems, Michael Barnsley presents two algorithms for computing the attractor of an IFS. The first ``deterministic algorithm'' constructs the attractor directly in iterated steps. The second ``random iteration algorithm'' (or ``chaos game'') plots hundreds of thousands of points, where each point is the image of a randomly selected transformation on the previous point, and the collection of points approximates the attractor of the IFS. In particular, Barnsley uses a computer's pseudorandom number generator to select the transformations. A famous attractor of an IFS is the Barnsley fern and is shown in Figure \ref{fig:barnsley1}.

\begin{figure}[h]
	\centering
	\includegraphics[width=0.5\textwidth]{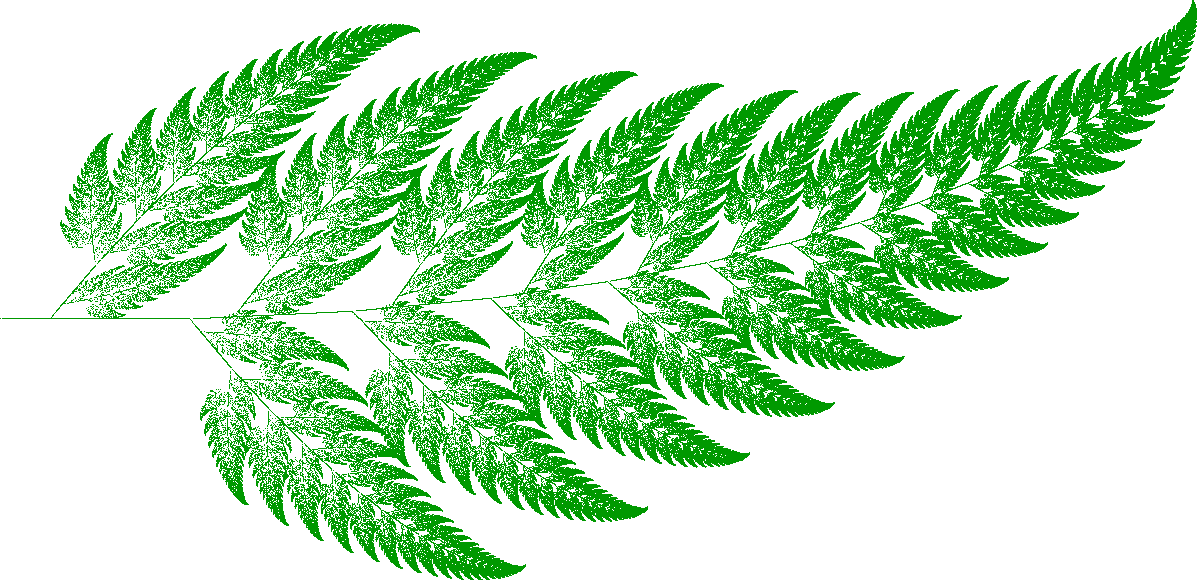}
	\caption{The Barnsley fern.}
	\label{fig:barnsley1}
\end{figure}

We begin by reintroducing iterated function systems (with probabilities) and the random iteration algorithm.

\subsubsection{An Note on Illustrations}
The illustrations appearing in this paper are the output of a program written in Processing by the author. It is important to note now that the illustrations are of plots in Cartesian coordinates, but with the convention that the origin $(0, 0)$ appears at the top-left of the image and with the $y$-axis increasing downwards rather than upwards. The $x$-axis increases to the right as usual. The source code for the program, including a Python version with a user interface, can be found at \cite{delapocode}.

\subsection{Iterated Function Systems}

\begin{definition}
	An \textit{iterated function system with probabilities} consists of a metric space $(X, d)$, a finite collection of transformations $f_1, f_2, \dotsc, f_n : X \to X$, and a corresponding collection of real probabilities $p_1, p_2, \dotsc, p_n$, where $0 < p_i < 1$ for all $i$, and $\sum_{i=1}^n p_i = 1$. An iterated function system with probabilities, often abbreviated IFS, is often presented as $\{X;\ f_1, f_2, \dotsc, f_n;\ p_1, p_2, \dotsc, p_n\}$. When the probabilities are omitted, one can assume that the probabilities are uniform, and $p_i = \frac{1}{n}$ for all $i$.
\end{definition}

\begin{definition}
	Let $(X, d)$ be a metric space. A transformation $f: X \to X$ is a \textit{contraction mapping} if there is a constant $0 \leq s < 1$ such that for all $x, y \in X$,
	\begin{equation*}
		d(f(x), f(y)) \leq s \cdot d(x, y)
	\end{equation*}
\end{definition}

\begin{definition}
	Let $\{X;\ w_1, w_2, \dotsc, w_n\}$ be an IFS where each $w_i$ is a contraction mapping. Barnsley calls such an IFS \textit{hyperbolic}. Let $\mathscr{H}(X)$ denote the space whose points are the compact subsets of $X$, not including the empty set. One can check (see \cite{barnsley}) that the transformation $W: \mathscr{H}(X) \to \mathscr{H}(X)$ defined by
	\begin{equation*}
		W(B) = \bigcup_{i=1}^n w_i(B)
	\end{equation*}
	has a unique fixed point $A \in \mathscr{H}(X)$; we have $W(A) = A$, and $A$ is given by
	\begin{equation*}
		A = \lim_{n \to \infty} W^n(B)
	\end{equation*}
	for any $B \in \mathscr{H}(X)$. Then $A$ is called the \textit{attractor} of the IFS.
\end{definition}

\begin{definition}
	One can use the \textit{random iteration algorithm} to approximate the attractor of an IFS $\{X;\ f_1, f_2, \dotsc, f_n;\ p_1, p_2, \dotsc, p_n\}$. The random iteration algorithm proceeds as follows.\\
		
	First, set $x_0 \in X$ arbitrarily. In cases where $X = \R^2$, we will set $x_0 = (0, 0)$. Next, for each $k \geq 1$, choose recursively and independently
	\begin{equation*}
		x_k \in \{f_1(x_{k-1}), f_2(x_{k-1}), \dotsc, f_n(x_{k-1})\}
	\end{equation*}
	where the probability that $x_k = f_i(x_{n-1})$ is $p_i$. The result of the random iteration algorithm is $\{x_n : n \in \N\} \subseteq X$. By ``randomly,'' Barnsley is referring to an unspecified level of randomness, but one that is at least as random as the pseudorandom number generator on a computer.
\end{definition}

\begin{example}\label{ex:sierpinski1}
	In $\R^2$, consider the three transformations
	\begin{align*}
		f_1(x, y) &= \left(\frac{x}{2}, \frac{y}{2}\right)\\
		f_2(x, y) &= \left(\frac{x}{2}, \frac{y + 100}{2}\right)\\
		f_3(x, y) &= \left(\frac{x + 100}{2}, \frac{y + 100}{2}\right)
	\end{align*}
	Then $f_1$ can be thought of as taking $(x, y)$ to the point halfway between itself and the origin. Similarly, $f_2$ takes $(x, y)$ halfway to $(0, 100)$, and $f_3$ takes $(x, y)$ halfway to $(100, 100)$. The result of the random iteration algorithm on the IFS $\left\{\R^2;\ f_1, f_2, f_3\right\}$ (where the probabilities are uniform) is a Sierpinski triangle, as seen in Figure \ref{fig:sierpinski1a}. On the right, we use probabilities $0.8$, $0.1$, and $0.1$ for $f_1$, $f_2$, and $f_3$ respectively, as seen in Figure \ref{fig:sierpinski1b}.
\end{example}

\begin{figure}[h]
	\centering
	\begin{subfigure}[t]{0.4\textwidth}
		\includegraphics[width=\textwidth]{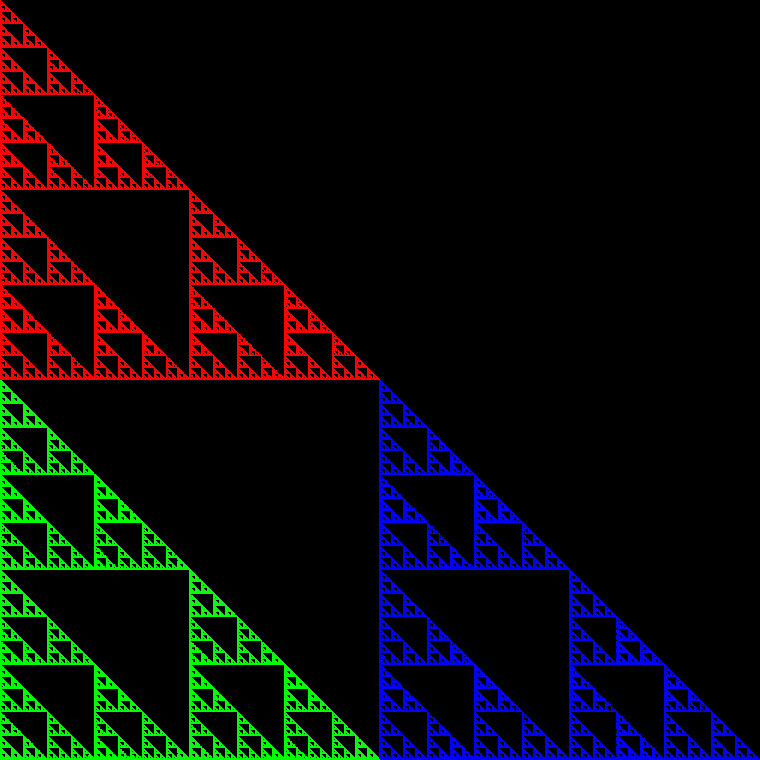}
		\subcaption{The result of one million iterations of random iteration algorithm on the IFS $\left\{\R^2;\ f_1, f_2, f_3;\ \frac{1}{3}, \frac{1}{3}, \frac{1}{3}\right\}$ from Example \ref{ex:sierpinski1} is the Sierpinski triangle, with vertices at $(0, 0)$, $(0, 100)$, and $(100, 100)$.}
		\label{fig:sierpinski1a}
	\end{subfigure}
	\hspace{0.2in}
	\begin{subfigure}[t]{0.4\textwidth}
		\includegraphics[width=\textwidth]{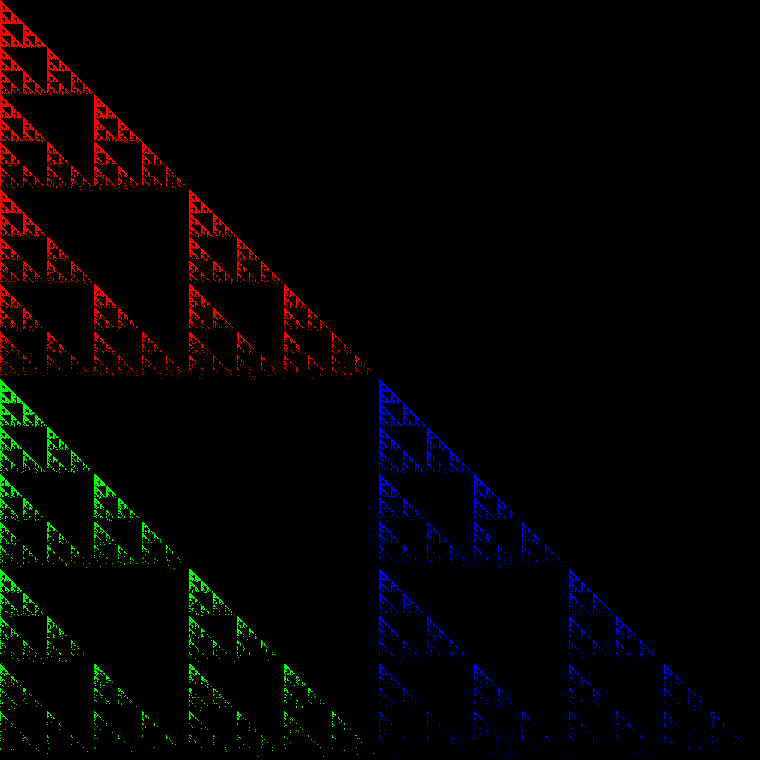}
		\subcaption{The result of one million iterations of the random iteration algorithm on the same IFS as in (a), except with probabilities $0.8$, $0.1$, and $0.1$ for $f_1$, $f_2$, and $f_3$, respectively.}
		\label{fig:sierpinski1b}
	\end{subfigure}
	\caption{Two results of the random iteration algorithm with the same transformations but different probabilities. In each picture, a color is associated to each transformation, so that $f_i(x, y)$ is given the color associated with $f_i$.}
	\label{fig:sierpinski1}
\end{figure}

\subsection{Randomness and Iterated Function Systems}
We modify the random iteration algorithm to instead use a pre-determined sequence to choose from the $n$ transformations at each step.

\begin{definition}
	Let $\{X;\ f_0, f_1, \dotsc, f_{n-1}\}$ be an IFS. Let $\sigma \in n^{\omega}$. The \textit{determined iteration algorithm} is the following modification of the random iteration algorithm. Pick $x_0 \in X$ arbitrarily as in the random algorithm, and pick $x_{k} = f_{\sigma[k - 1]}(x_{k-1})$ for each $k \geq 1$. The result of the determined iteration algorithm is $\{x_k : k \in \N\}$.
\end{definition}

\begin{example}\label{ex:unitSquare1}
	Let $v_0 = (0, 0), v_1 = (0, 1), v_2 = (1, 0), v_3 = (1, 1) \in \R^2$, and consider the IFS $\{\R^2, f_0, f_1, f_2, f_3\}$, where each $f_i$ is the midpoint transformation from $(x, y)$ to the point halfway between $(x, y)$ and $v_i$. The attractor of this IFS is the unit square, and when the probability of each $f_i$ is $p_i = \frac{1}{4}$, the square is uniformly covered with points when the random iteration algorithm is applied, as in Figure \ref{fig:unitSquareRandom1}. Champernowne's base 4 sequence produces the result in Figure \ref{fig:unitSquareChampernowne1}. Because the first 15 digits of $C_4$ are
	\begin{equation*}
		012310111213202
	\end{equation*}
	the first 15 transformations chosen in the determined iteration algorithm are, in order,
	\begin{equation*}
		f_0, f_1, f_2, f_3, f_1, f_0, f_1, f_1, f_1, f_2, f_1, f_3, f_2, f_0, f_2
	\end{equation*}
	
	\begin{figure}[h]
		\centering
		\begin{subfigure}[t]{0.28\textwidth}
			\includegraphics[width=\textwidth]{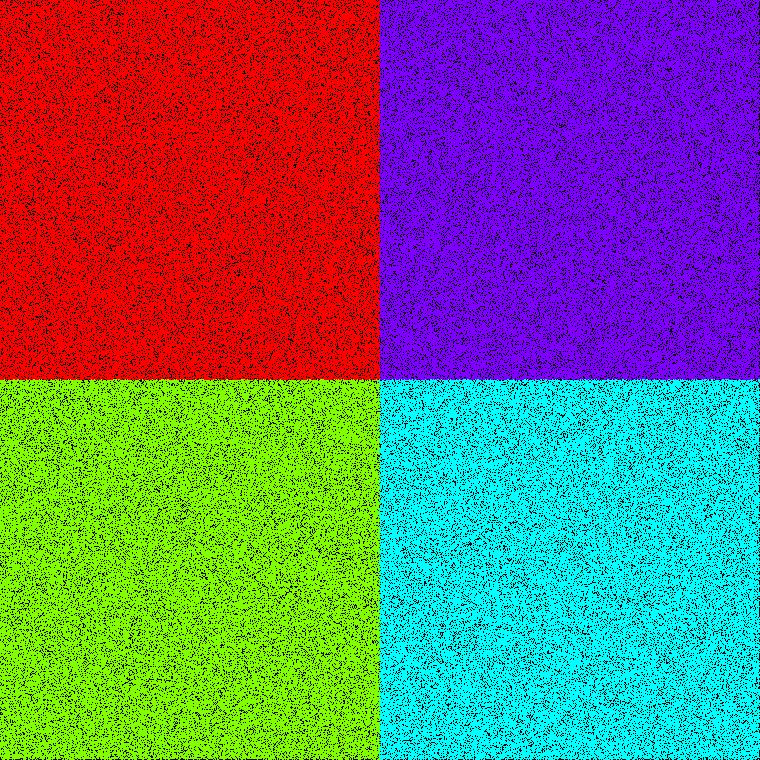}
			\subcaption{A result of one million iterations of the random iteration algorithm on the IFS $\{\R^2, f_0, f_1, f_2, f_3\}$ from Example \ref{ex:unitSquare1} using a pseudo-random number generator.}
			\label{fig:unitSquareRandom1}
		\end{subfigure}
		\hspace{0.02\textwidth}
		\begin{subfigure}[t]{0.28\textwidth}
			\includegraphics[width=\textwidth]{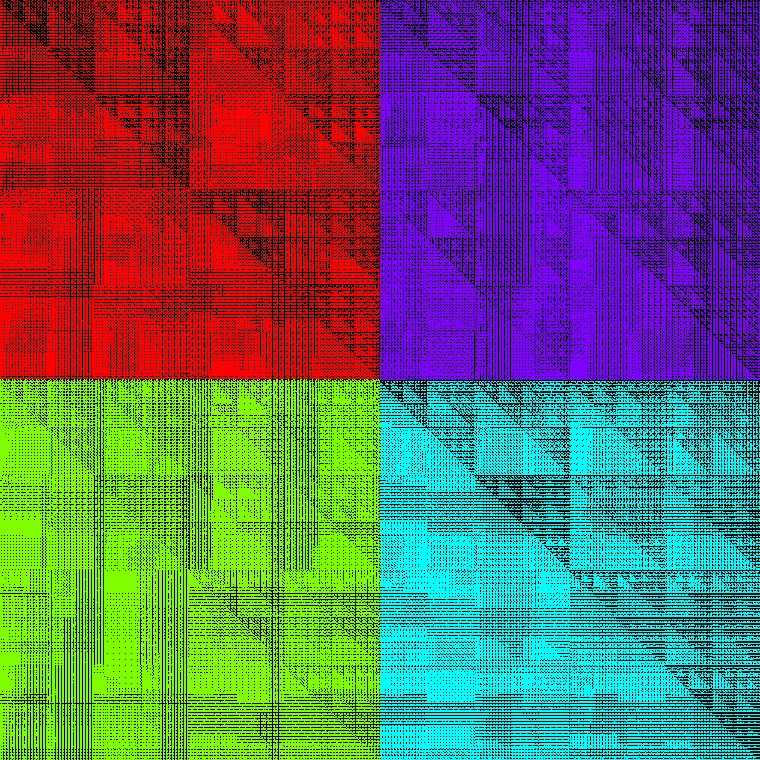}
			\subcaption{The result of one million iterations of the determined iteration algorithm on the same IFS as in (a). The transformations were determined by $C_4$.}
			\label{fig:unitSquareChampernowne1}
		\end{subfigure}
		\hspace{0.02\textwidth}
		\begin{subfigure}[t]{0.28\textwidth}
			\includegraphics[width=\textwidth]{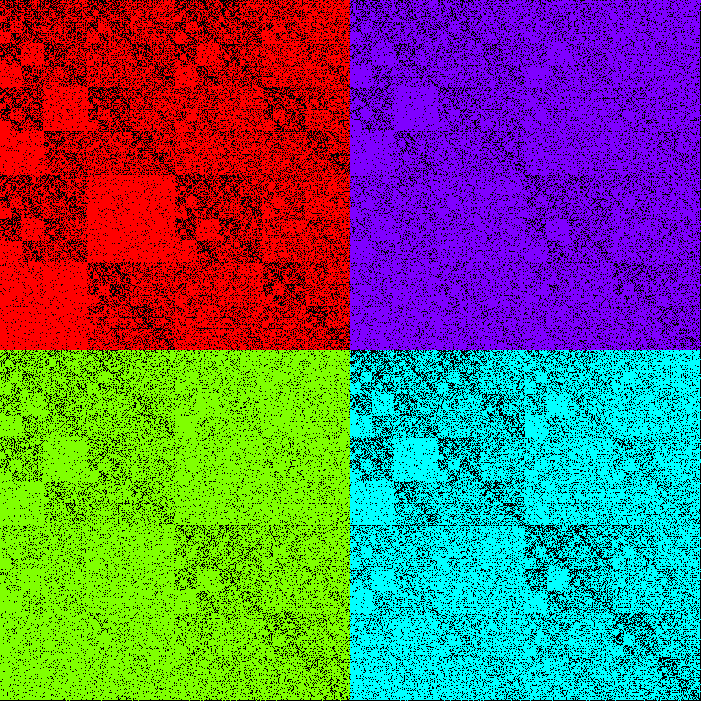}
			\subcaption{The result of one million iterations of the determined iteration algorithm on the same IFS as in (a). The transformations were determined by $CE_{4}$.}
			\label{fig:unitSquareCE1}
		\end{subfigure}
		\caption{Comparing the random iteration algorithm with the determined iteration algorithm.}
		\label{fig:unitSquare1}
		
	\end{figure}
\end{example}

By the definition of \textit{normal}, each transformation has the same chance of being applied to $x_n$ as every other transformation. Not all iterated function systems use uniform probabilities, however. Barnsley's fern, for example, uses four affine transformations with probabilities $0.85$, $0.07$, $0.07$, and $0.01$. This motivates the definition and construction of biased normal sequences.

\section{Further Questions}\label{sec:questions}

\begin{enumerate}[(1)]
	\item Let $\mu$ be a Borel probability measure, $x \in [0, 1]$ a real number, and $b$ a base. For each positive integer $n$ and interval $I \subseteq [0, 1]$, let
	\begin{equation*}
		f_I(n, x) = \card{\{k \in \Z : \text{$1 \leq k \leq n$ and there exists $y \in I$ such that $b^k x \equiv y\ \text{mod}\ 1$}\}}.
	\end{equation*}
	Say that $x$ is \textit{$\mu$-normal} if for every interval $I \subseteq [0, 1]$,
	\begin{equation*}
		\lim_{n \to \infty} \frac{f_I(n, x)}{n} = \mu(I).
	\end{equation*}
	
	What are the necessary and sufficient conditions on $\mu$ such that every $\mu$-Martin-L\"of-random real $x$ is $\mu$-normal?
	
	\item One can consider the set of bases to which a given real number is normal, and conversely one can ask whether there exists a real number which is normal to a set of bases. Similar questions can be asked in the biased case. For example, suppose $x$ is a Bernoulli random real in base $b$. For every base $b'$ multiplicatively independent of $b$, do there exist densities to which $(x)_{b'}$ is biased normal? If not, give a counterexample. For published progress on this question for the case of uniform biases, see \cite{bugeaud}. Preliminary investigations suggest that the assumption of Bernoulli randomness cannot be weakened to biased normality, since it appears that there exist reals which are biased normal for all bases multiplicatively independent of $b = 3$ but not biased simply normal in base $3$.
	
	\item Do biased normal reals compute normal reals? If so, does this algorithm also compute a $\lambda$-Martin-L\"of random real given Bernoulli random real? In \cite{porter}, Porter states that von Neumann's randomness extractor achieves the desired result for binary sequences.
	
	\begin{conjecture}
		There is a generalization of von Neumann's randomness extractor which computes normal reals from biased normal reals and $\lambda$-Martin-L\"of random reals from Bernoulli random reals.
	\end{conjecture}
	
	\item What are the necessary and sufficient conditions for a real number, using the determined iteration algorithm, to generate the same attractor as the random iteration algorithm? This question can be formalized using the results presented by Barnsley in \cite{barnsley}.
	
	Assume that $(X, d)$ is a compact metric space and $\{X; w_0, w_1, \dotsc, w_n; p_0, p_1, \dotsc, p_{n-1}\}$ is a hyperbolic IFS with probabilities. By Theorems 9.6.1 and 9.6.2 of \cite{barnsley}, there is a unique normalized Borel measure $\nu$ on $X$ associated with the IFS such that the support of $\nu$ is the attractor of the IFS. The measure $\nu$ is called the \textit{invariant measure} associated with the IFS. If $\{x_k : k \in \N\}$ is the result of the determined iteration algorithm using $\sigma$, then let
	\begin{equation*}
		\mathcal{N}(B, n) = \card{\{x_0, x_1, \dotsc, x_n\} \cap B}
	\end{equation*}
	for any Borel subset $B$ of $X$. Let $S$ be the set of sequences $\sigma$ in $n^{\omega}$ which, under the determined iteration algorithm, will satisfy
	\begin{equation*}
		\nu(B) = \lim_{n \to \infty} \frac{\mathcal{N}(B, n)}{n + 1}
	\end{equation*}
	for every Borel subset $B$ of $X$ with measure $0$ boundary. By Corollary 9.7.1 of \cite{barnsley}, if the parameters of the Bernoulli measure are the probabilities $p_0, p_1, \dotsc, p_{n-1}$ from the IFS, then $S$ has Bernoulli measure $1$.
	
	In summary, $S$ is the set of sequences $\sigma \in n^{\omega}$ such that in the determined iteration algorithm using $\sigma$, each Borel subset $B$ with null boundary is visited with the frequency given by $\nu(B)$. What randomness properties must $\sigma$ have such that $\sigma \in S$? One can further ask if there a connection between the discrepancy of $\sigma$ and the rate at which the determined iteration algorithm approximates the attractor produced by the random iteration algorithm.
	
	\begin{conjecture}
		Given a hyperbolic IFS with probabilities, a sequence $\sigma$ is an element of $S$ --- that is, $\sigma$ generates the attractor of the IFS as described above --- if and only if $\sigma$ is biased normal with respect to the probabilities of the IFS.
	\end{conjecture}
\end{enumerate}

\section{Acknowledgments}
This honors thesis was advised by Professor Theodore Slaman. I am grateful for Professor Slaman's time, guidance, and patience. His patience in helping me develop the proof of Lemma \ref{lem:biasednormalityalt1} is particularly noteworthy.

Conversations with Druv Pai about the binomial distribution and probability were helpful in developing the proofs of Lemmas \ref{lem:multinom} and \ref{lem:badblocks}.

For their support of the undergraduate mathematics community at UC Berkeley, I dedicate this senior thesis to Berkeley's Mathematics Undergraduate Student Association.

\bibliographystyle{unsrt}
\bibliography{ref}

\end{document}